\documentclass[pdflatex,sn-mathphys-num]{sn-jnl}% Math and Physical Sciences Numbered Reference Style
\usepackage[utf8]{inputenc}
\usepackage{wrapfig}
\usepackage{comment}
\usepackage[T1]{fontenc}
\usepackage{amsmath,amssymb,amsthm}
\usepackage{mathtools}
\usepackage{booktabs}
\usepackage{array}
\usepackage{multirow}
\usepackage{graphicx}
\usepackage{hyperref}
\usepackage[capitalise,noabbrev]{cleveref}
\usepackage{enumitem}
\usepackage{xcolor}
\usepackage{caption}
\usepackage{tikz}
\usetikzlibrary{fit,backgrounds,decorations.pathreplacing}

\newtheorem{theorem}{Theorem}
\newtheorem{proposition}[theorem]{Proposition}
\newtheorem{definition}{Definition}
\newtheorem{corollary}{Corollary}
\newtheorem*{remark}{Remark}

\usepackage{lmodern}
\usepackage[T1]{fontenc}
\usepackage{placeins}
\usepackage{xurl}
\hypersetup{pdftitle={Scaling Local Solvers with the Many-Body Expansion},
            pdfauthor={Federico Zahariev and Amr Sabry}}

\newcommand{\R}{\mathbb{R}}
\newcommand{\MBE}{\mathrm{MBE}}
\newcommand{\abs}[1]{\lvert #1 \rvert}

\newcommand{\fzadd}[1]{#1}
\newcommand{\fznote}[1]{}
\newcommand{\fzdata}[1]{#1}
\newcommand{\fzfix}[1]{#1}
\newcommand{\fzdom}[1]{#1}
\newcommand{\fzrev}[1]{#1}

\begin{document}

% ---- Title ----
\title{Scaling Local Solvers with the Many-Body Expansion}
%\title{Exact Decomposition of Combinatorial Optimization Problems via the Many-Body Expansion}

\author*[1,2]{\fnm{Federico} \sur{Zahariev}}\email{fzahari@iastate.edu}
\author[3]{\fnm{Amr} \sur{Sabry}}\email{sabry@iu.edu}
\equalcont{These authors contributed equally to this work.}
\affil*[1]{\orgdiv{Department of Chemistry}, \orgname{Iowa State University}}
\affil[2]{\orgname{Ames Laboratory}}
\affil[3]{\orgdiv{Department of Computer Science}, \orgname{Indiana University}}

% ===========================================================================

\abstract{
We introduce a Many-Body Expansion framework for scaling black-box solvers
for discrete optimization. The method decomposes a large problem into
fragments, repeatedly improves each fragment with the rest of the assignment
held fixed, and then applies targeted corrections on unions of fragments.
The solver itself is unchanged: it only receives a restricted subproblem and
returns a candidate assignment.

The framework separates decomposition from solver design. M\"obius inversion
gives an exact fragment expansion of the objective, and for pairwise
objectives such as graph coloring, MAX-CUT, and QUBO the expansion closes at
second order. The promoted solver is monotone by construction, \fzfix{terminates in
finitely many productive sweeps for non-negative integer-valued objectives,}
and connects to junction-tree dynamic
programming when the fragment interaction structure is acyclic \fzfix{and
full response functions are propagated}.

On graph coloring benchmarks, the promoted versions of TabuSearch and
simulated annealing (SA) \fzrev{generally outperform flat applications of the same
solvers, often finding better solutions in less CPU time, and the gain proves
insensitive to the choice of graph partition. A facility-location study
delimits the claim: there the partition must respect the objective's
interaction locality, and the promoted solver reaches within 0.3--1.4\% of
GLPK-certified optima at every tested size.} \fzfix{The same wrapper interface is used in each solver-swap case, and the
facility-location check reuses the same monotone guard with a sign-adjusted
interaction diagnostic; the improvements therefore come from the MBE
promotion mechanism rather than solver-specific tuning.}
}


\maketitle

% ===========================================================================
\section{Introduction}

Many optimization methods work well locally. A heuristic may quickly improve a
small graph coloring instance, a quantum routine may handle a modest QUBO, and
an exact solver may be practical on a bounded-size subproblem. The difficulty
is scale: success on small pieces does not by itself say how to coordinate
those pieces into an efficient high-quality solution of the full problem.

This paper develops a general promotion mechanism for such local solvers. The
input is any solver that can improve a restricted subproblem while the rest of
the assignment is held fixed. The promotion partitions the variables into
fragments, \fzfix{cycles through the fragments until a sweep produces no
strict objective decrease,} and then applies higher-order corrections to
selected unions of fragments. The solver remains a black box throughout. It does not need to know
how the global objective is decomposed, how fragments interact, or whether it
is being used as part of a larger algorithm.

\begin{wrapfigure}[17]{r}{0.4\textwidth}
\vspace{-1em}
\begin{tikzpicture}[
    every node/.style={circle, draw, minimum size=8mm, font=\small}
  ]
  \node (1) at (0,  1.5) {1};
  \node (2) at (-1.2, 0) {2};
  \node (3) at ( 1.2, 0) {3};
  \node (4) at (0, -1.5) {4};
  \draw (1)--(2); \draw (1)--(3); \draw (2)--(3);
  \draw (2)--(4); \draw (3)--(4);
  \node[draw=none, font=\small] at (-1.7, 0.95) {$F_1 = \{1,2\}$};
  \node[draw=none, font=\small] at ( 1.7,-0.95) {$F_2 = \{3,4\}$};
  \node[draw=none, font=\footnotesize] at (0.15, 0.25) {$\Delta$};
\end{tikzpicture}
\caption{A two-fragment coloring problem. The delta value counts boundary
conflicts between fragments; a joint solve of $F_1\cup F_2$ measures whether
those fragments can still improve together.}
\label{fig:diamond}
\end{wrapfigure}
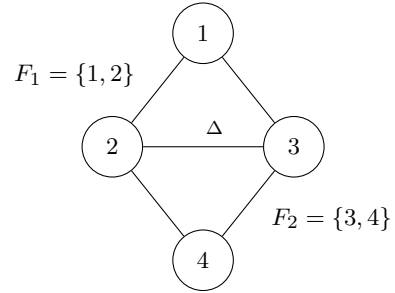

The small graph in \cref{fig:diamond} shows the basic mechanism. Suppose
$F_1=\{1,2\}$ and $F_2=\{3,4\}$. A local solver first improves the coloring
of $F_1$ while the colors in $F_2$ are held fixed, then improves $F_2$ with
$F_1$ fixed, and repeats until neither local call improves the objective.
\fzfix{At that point the pair is objective-stable: neither restricted call
has returned a strictly improving candidate against the other's current
values.} Each
edge crossing the fragment boundary contributes to the delta value: if any of
them is monochromatic, the two fragments remain in conflict and are natural
candidates for a joint correction. Re-solving $F_1\cup F_2$ then measures the
actual residual improvement available to the pair. Delta identifies where to
look; the residual says whether a move on that union can still help.

The mathematical structure behind this promotion is the Many-Body Expansion
(MBE), a decomposition principle from computational quantum chemistry
\cite{gordon2012fragmentation,fedorov2007fmo}. In chemistry, the energy of a
large molecular system is organized into contributions from individual
fragments, pairs of fragments, triples, and so on. The bookkeeping is
governed by M\"obius inversion on the lattice of fragment subsets. We use the
same structure for discrete optimization: the objective is decomposed into
fragment, pair, and higher-order terms, while solver calls are used to improve
assignments on the corresponding fragment unions.

For pairwise objectives such as graph coloring, MAX-CUT, and QUBO, the MBE
expansion closes at second order: fragment and pair terms recover the
objective value exactly. Optimization still requires coordination among
fragments, but the same hierarchy supplies the coordination mechanism. The
first stage performs self-consistent local improvement. Later stages revisit
pairs, triples, and larger fragment unions when the diagnostics indicate that
their interaction still matters. Every proposed update is accepted only when
it does not worsen the objective, so the promoted solver is monotone by
construction. \fzfix{For non-negative integer-valued objectives, the
self-consistency stage terminates after finitely many productive sweeps.}

This perspective also connects MBE to classical decomposition methods. When
the fragment interaction structure is acyclic, exact coordination is achieved
by message passing along the tree, matching the dynamic-programming principle behind junction-tree methods. 
For general interaction graphs, the same MBE
view gives a practical correction hierarchy: local solvers handle the
fragments, delta values prioritize boundary conflicts, and residual checks
certify when a chosen fragment union is exhausted with the exterior fixed.

We evaluate the framework on graph coloring and MAX-CUT benchmarks using
structurally different solvers, including TabuSearch and simulated
annealing (SA).
The same MBE pipeline wraps each solver; only the fragment-level solver
function changes. \fzrev{The promoted solvers outperform flat applications of the same solvers
on most tested instances, often finding better assignments in less CPU time;
the exceptions (SA on two instances, and spatially random partitions in the
facility-location study) are reported and delimit the claim.} The experiments suggest that the benefit
comes from the decomposition and correction structure itself, not from
solver-specific tuning.

% ===========================================================================
\section{The MBE Optimization Algorithm}

This section describes MBE as an algorithmic wrapper around an arbitrary
local solver. The formal expansion and its guarantees are given in the next
section; here the goal is to show how the promoted solver is actually run.

Let $f\colon X_1\times\cdots\times X_n \to \R$ be an objective over
discrete variables, with lower values preferred. A solver call is
restricted to a subset of variables. Given a fragment $F\subseteq[n]$
and a current global assignment $x$, the solver returns a candidate
assignment $x'$ that agrees with $x$ outside~$F$. The solver may be a
heuristic, an exact solver, a quantum routine, or a brute-force
search; MBE treats it as a black box.

In the self-consistency stage, every solver call is wrapped by the same
non-worsening acceptance rule:
\[
  x \leftarrow
  \begin{cases}
    x' & \text{if } f(x') \le f(x),\\
    x  & \text{otherwise,}
  \end{cases}
\]
\fzfix{and cycling terminates after the first sweep that produces no
\emph{strict} decrease of $f$. Equal-cost proposals may thus be accepted
(they can move boundary values at no objective cost), but they do not count
as progress: the finite-convergence guarantee in \cref{thm:convergence}
counts only strictly improving sweeps, so accepting ties affects neither the
bound nor monotonicity. The same rule is used unchanged in the one-shot
correction passes.}

The wrapped solver may restart, accept uphill moves internally, or use any
search strategy it likes. \fzfix{The global assignment never worsens: a
proposed fragment update is adopted only when it does not increase the full
objective.}

\paragraph{Stage~1: self-consistent fragment cycling.}
Partition the variables into disjoint fragments
$F_1,\ldots,F_k$. Starting from an initial assignment $x^{(0)}$, MBE
sweeps through the fragments. On fragment $F_i$, the variables outside
$F_i$ are held fixed and the solver is asked to improve only the
variables inside~$F_i$. After each call, the keep-if-not-worse rule
decides whether the global assignment is updated.

A sweep is one pass through all fragments. \fzfix{Sweeps continue until a
full pass produces no strict decrease of the objective (equal-cost updates
may still be accepted during that pass). The terminal assignment is
objective-stable relative to the solver calls made in the final sweep: none
of those calls returned a strictly improving update against its then-current
boundary values.} For non-negative integer-valued objectives, this stage has
finite descent: each productive sweep decreases the objective by at least
one, so only finitely many productive sweeps are possible.

The main structural parameter is the number of fragments $k$. Large fragments
give the local solver more context but make each call harder; small fragments
make calls cheaper but push more structure into the boundaries. The
implementation may also impose solver-specific iteration caps, random seeds,
cooling schedules, or tabu-tenure parameters, but these are properties of the
chosen local solver rather than part of the MBE interface itself.

\paragraph{Stage~2: diagnostics and higher-order corrections.}
After Stage~1, MBE looks for interactions that were not resolved by
single-fragment updates. For pairwise objectives, the most direct diagnostic
is the \emph{delta spectrum}. For two fragments $F_i$ and $F_j$,
$\Delta(F_i,F_j)$ is the contribution to the current objective from terms
that cross between them. In graph coloring, it is simply the number of
monochromatic edges with one endpoint in $F_i$ and the other in $F_j$.

Large delta values identify fragment pairs whose boundary interaction is
still costly. MBE therefore applies dimer corrections first to pairs with
positive or large delta:
\[
  F_i,\ F_j
  \quad\leadsto\quad
  F_i\cup F_j .
\]
The solver is called on the union while all variables outside the union are
held fixed, and the keep-if-not-worse rule is applied again. The same idea can
be continued to triples and larger unions when pair corrections do not
remove the remaining conflict.

Delta is a guide for choosing which unions to revisit. The corresponding
certificate is the residual of the union: how much the objective can still
be improved by changing that union while the exterior is frozen. If an exact
solve of $U$ shows zero residual, then no move confined to $U$ can improve
the current assignment. With heuristic solvers, the same quantity is used
operationally as an observed improvement signal rather than a formal
certificate.

\paragraph{Merging persistent interactions.}
Sometimes a small number of fragment pairs continue to carry most of the
residual conflict after dimer and trimer corrections. In that case MBE can
merge the highest-conflict pair permanently into a larger fragment and
restart the self-consistency stage with the new partition. Merging increases
the cost of later local solves, but it gives the solver more freedom in the
part of the problem where the current decomposition is too restrictive.

\paragraph{Graph coloring specialization.}
For graph coloring with $q$ colors on $G=(V,E)$, the variables are
colors $x_v\in\{0,\ldots,q-1\}$ and $f(x)=|\{(u,v)\in E : x_u=x_v\}|$
counts monochromatic edges. A fragment is a subset of vertices. A
solver call on $F$ changes only the colors of vertices in $F$,
treating adjacent vertices outside $F$ as fixed boundary values.

\fzfix{In this setting, Stage~1 repeatedly recolors fragments until a
complete sweep produces no strict objective decrease; equal-cost recolorings
may be accepted but are not counted as progress.} Stage~2 computes the crossing conflicts
between fragment pairs and re-solves selected unions of fragments. The same
pipeline applies to TabuSearch, SA, DSATUR, or any other
solver that can be restricted to a subset of vertices. MAX-CUT is obtained
by taking $q=2$ and changing the objective from monochromatic edges to cut
violations; the MBE wrapper itself is unchanged.

% ===========================================================================
\section{The Many-Body Expansion: Formal Guarantees}
\label{sec:mbe}

The algorithm of the previous section makes three implicit promises:
that cycling terminates, that the correction hierarchy never undoes
progress, and that the delta diagnostics point at real
improvements. This section makes those promises precise and identifies the
exact conditions under which the second-order expansion is exact rather than
heuristic.

Let $f\colon X_1\times\cdots\times X_n \to \R$ be an objective with an
additive decomposition
\[
  f(x)=\sum_{\alpha\in A}\phi_\alpha(x_\alpha),
\]
where each term $\phi_\alpha$ depends on a subset of variables. Partition
$[n]$ into $k$ disjoint fragments $F_1,\ldots,F_k$. The fragment support of
a term is
\[
  \mathrm{supp}_{\mathcal F}(\alpha)
  := \{i\in[k] : \phi_\alpha \text{ depends on at least one variable in } F_i\}.
\]
For $S\subseteq[k]$, define
\[
  f_S(x)
  := \sum_{\alpha:\,\mathrm{supp}_{\mathcal F}(\alpha)\subseteq S}
     \phi_\alpha(x_\alpha),
\]
with $f_\emptyset=0$ when there is no constant term.

\begin{definition}\label{def:mbe}
The \emph{cluster term} for $S\subseteq[k]$ is
\[
  C_S \;:=\; \sum_{T\subseteq S}(-1)^{\abs{S}-\abs{T}}\,f_T.
\]
The \emph{MBE at order~$p$} is the partial sum
\[
  f^{(p)} \;:=\; \sum_{\substack{S\subseteq[k]\\1\le\abs{S}\le p}} C_S.
\]
At full order, $f^{(k)}\equiv f$.
\end{definition}

\begin{remark}
The cluster term $C_S$ is the M\"obius inversion of the family
$\{f_T\}_{T\subseteq S}$ on the Boolean lattice of fragment subsets. The
alternating sign acts on the restricted objectives $f_T$, not on $f_S$ alone.
This distinction matters: summing only scalar signs would give zero for every
nonempty cluster by the binomial identity.
\end{remark}

% ---------------------------------------------------------------------------
\subsection{Finite Convergence}

\begin{theorem}\label{thm:convergence}
\fzfix{For non-negative integer-valued objectives, Stage~1 with non-worsening
acceptance terminates after at most $M=f(x^{(0)})$ \emph{productive} sweeps
(sweeps containing at least one strictly improving accepted update) plus one
terminal sweep. The terminal sweep certifies that no solver call made in it
returned a strictly improving update against the boundary values it was
given.}
\end{theorem}

\begin{proof}
\fzfix{The objective never increases, since only non-worsening updates are
accepted. Every productive sweep decreases the integer-valued objective by at
least~1; since the objective is non-negative and starts at $M=f(x^{(0)})$,
there are at most $M$ productive sweeps. Cycling stops at the first
non-productive sweep, so at most $M+1$ sweeps occur in total. In the
terminal sweep, no solver call returned a candidate with strictly lower
full-objective value: equal-cost candidates may have been accepted, worse
candidates were rejected, and no strict decrease occurred during the sweep,
which is exactly the stated property.}
\end{proof}

\begin{remark}
ADMM was designed for continuous convex optimization and does not in
general carry finite-convergence guarantees when adapted to discrete
problems, where it can oscillate indefinitely.
MBE's finite convergence is therefore a structural advantage for
integer-valued objectives.
In our runs (reproducible via \texttt{--sweep-count}), Stage~1 reached its
termination certificate in $1$--$9$ sweeps on the DIMACS instances and
within~$12$ sweeps on all but the largest generated instance; on
rand-500-0.03 ($k=39$) some seeds were still improving when the 15-sweep
experimental cap (Appendix~\ref{app:repro}) truncated cycling. The cap is in
every case far below the $M$ upper bound, and truncation is benign for
correctness: the objective is monotone throughout, and the correction passes
of Stages~2--3 operate on whatever Stage~1 delivers.
\fzfix{Two caveats delimit the certificate. First, it is relative to the
solver calls actually made: with stochastic or incomplete local solvers, a
non-productive sweep shows only that those calls found no strict improvement,
not that none exists. Second, equal-cost updates accepted during the terminal
sweep can move boundary values, so the terminal assignment is stable in
objective value but not necessarily a fixed point of the assignment itself;
the correction passes of Stages~2--3 exploit exactly such situations.}
\fzadd{For finite-domain real-valued objectives, such as Ising or QUBO
instances with real coefficients, \fzfix{the same non-worsening rule with
strict-progress termination still terminates, because the loop continues only
on a strict decrease and the number of distinct accepted objective values is
finite.}
However, the sharper $M=f(x^{(0)})$ descent bound is then replaced by the
number of distinct accepted objective values, or more coarsely by the finite
state-space size. Thus the integer theorem is the strong bound used for
conflict-count objectives, while the real-valued experiments test the same
monotone promotion mechanism without claiming that integer bound.}
\end{remark}

% ---------------------------------------------------------------------------
\subsection{Locality and the Delta Spectrum}

After Stage~1 convergence to $x_{\MBE}$, a pairwise objective decomposes as
\[
  f(z) \;=\; \sum_i I_i(z_i) + \sum_{i<j} B_{ij}(z_i,z_j),
\]
where $I_i\ge 0$ collects the pairwise terms internal to $F_i$ and
$B_{ij}\ge 0$ the terms coupling $F_i$ and $F_j$ (for conflict-count
objectives such as coloring, these are the violated edges within and across
fragments, but the decomposition is generic).
The \emph{delta spectrum} is $\Delta(F_i,F_j):=B_{ij}(x_i,x_j)$, the
current cross-fragment interaction cost for each pair.

For any union of fragments $U$, let $A_U(z_U;\,x_{\bar U})$ be the sum
of all terms in $f$ with at least one endpoint in $U$, with exterior
variables frozen at $x_{\bar U}$.
The \emph{dimer residual} for $U$ is
\[
  g_U(x) \;:=\; A_U(x_U;\,x_{\bar U}) - \min_{z_U} A_U(z_U;\,x_{\bar U}).
\]

\begin{proposition}[Locality of corrections]\label{prop:locality}
For any assignment $y$ agreeing with $x$ outside $U$,
\[
  f(x) - f(y) \;\le\; g_U(x),
\]
with equality when $y_U$ minimizes $A_U(\cdot;\,x_{\bar U})$.
In particular, $g_U(x)=0$ certifies that no move confined to $U$ can
improve~$f$.
\end{proposition}

\begin{proof}
Since $y$ agrees with $x$ outside $U$, the terms not involving $U$ cancel:
$f(x)-f(y) = A_U(x_U) - A_U(y_U) \le A_U(x_U) - \min_{z_U}A_U(z_U) = g_U(x)$.
\end{proof}

\begin{remark}
  The dimer residual $g_U$ is the correct certificate: $g_U(x)=0$
  certifies that no joint re-optimization of $U$ can improve $f$ given
  the current exterior. The delta spectrum $\Delta(F_i,F_j)$ is a
  useful \emph{diagnostic} for identifying pairs with high boundary
  frustration, but $\Delta_{\mathrm{total}}=0$ does not imply $g_U=0$
  for all dimers: an intra-fragment improvement may be unlocked by
  changing a boundary node in the other fragment, as a counterexample
  demonstrates \fzrev{(Appendix~\ref{app:counterexamples},
    Construction~1)}.  The delta spectrum guides which pairs to
  re-solve first; the residual $g_U$ certifies whether re-solving
  helped.
\end{remark}

% ---------------------------------------------------------------------------
\subsection{Monotonicity}

\begin{theorem}\label{thm:monotonicity}
Let $x^{(p)}$ be the assignment obtained after applying all MBE corrections
up to order~$p$, with every proposed update accepted only when it does not
increase the full objective. Then
\[
  f(x^{(1)}) \;\ge\; f(x^{(2)}) \;\ge\; \cdots \;\ge\; f(x^{(k)}).
\]
\end{theorem}

\begin{proof}
Each order-$(p{+}1)$ correction starts from $x^{(p)}$ and proposes updates
on unions of $p{+}1$ fragments. By the acceptance rule, each accepted update
has objective value at most the current objective value, and each rejected
update leaves the assignment unchanged. Therefore the objective cannot
increase when moving from order~$p$ to order~$p{+}1$.
\end{proof}

\begin{remark}
This is an internal monotonicity statement about one MBE run: higher-order
corrections cannot worsen the assignment produced by lower-order corrections.
It does not assert that MBE must beat an independent flat run of the same
heuristic. The comparison with flat solvers is an empirical question studied
in \cref{sec:experiments}.
\end{remark}

% ---------------------------------------------------------------------------
\subsection{Exact Truncation}

\begin{theorem}\label{thm:truncation}
Under \cref{def:mbe}, if every term in~$f$ depends on variables from at
most~$r$ fragments, then $f^{(p)}\equiv f$ for all $p\ge r$.
\end{theorem}

\begin{proof}
Consider one objective term $\phi_\alpha$ with fragment support
$R=\mathrm{supp}_{\mathcal F}(\alpha)$. This term contributes to $f_T$ if and
only if $R\subseteq T$. Its coefficient in $f^{(p)}$ is
\[
  \sum_{\substack{S:\,\abs{S}\le p}}
  \sum_{\substack{T:\,R\subseteq T\subseteq S}}
  (-1)^{\abs{S}-\abs{T}} .
\]
For a fixed $S$, the inner sum is zero unless $S=R$, by the binomial
identity applied to $S\setminus R$; when $S=R$, the inner sum is~1. Thus
$\phi_\alpha$ appears with coefficient~1 exactly when $\abs{R}\le p$, and
with coefficient~0 otherwise. Therefore all terms whose fragment support has
size at most~$r$ are included exactly once in every $f^{(p)}$ with $p\ge r$.
\end{proof}

\begin{corollary}\label{cor:truncation}
Graph coloring, MAX-CUT, QUBO: $r=2$ (pairwise), so $f^{(2)}\equiv f$.
MAX-SAT with clause width~$w$: $r=w$.
\end{corollary}

\begin{remark}
Exact truncation is a statement about \emph{evaluation}: the order-2 sum
exactly represents the objective value for any fixed assignment.
It does not imply that an algorithm which re-solves dimers sequentially
reaches the global optimum; that stronger claim requires message passing
along an acyclic intersection graph, as the next result establishes.
\end{remark}

% ---------------------------------------------------------------------------
\subsection{The Evaluation--Optimization Gap}

\begin{definition}\label{def:intersection}
The \emph{fragment-intersection graph} $I(G,\mathcal{F})$ has vertices~$[k]$
and an edge $(i,j)$ iff some edge of~$G$ joins $F_i$ to~$F_j$.
\end{definition}

\begin{theorem}\label{thm:gap}
Let $f$ be a pairwise objective.

\textup{(a)} If $I(G,\mathcal{F})$ is a forest, exact minimization is
achieved by order-2 message passing along the forest: the message from
fragment~$i$ to its parent~$j$ is the response function
$m_{i\to j}(x_j)=\min_{x_i}[I_i(x_i)+B_{ij}(x_i,x_j)
+\sum_{\ell\in\mathrm{ch}(i)}m_{\ell\to i}(x_i)]$,
which retains all of~$i$'s optimal responses rather than committing to a
single state. Minimizing at the root and backtracking recovers a globally
optimal assignment.

\textup{(b)} Independently committed dimer re-optimizations (the Stage~2
dimer corrections) are \emph{not} globally exact in general, even when
$I(G,\mathcal{F})$ is a tree: a fragment shared by multiple dimers can
require a cooperative move invisible to every individual dimer residual.

\textup{(c)} If $I(G,\mathcal{F})$ has a cycle of length $\ge 3$, there
exist instances where independent dimer optimizations produce globally
inconsistent boundary assignments.
\end{theorem}

\begin{proof}[Proof sketch of (a)]
Standard tree dynamic programming: each message $m_{i\to j}(\cdot)$ accounts
for the optimal cost of the subtree rooted at~$i$ as a function of the
parent's state. The running-intersection structure of the forest ensures
every edge of~$G$ appears in exactly one message computation.

Proof of (b): see the star-graph counterexample \fzrev{in
Appendix~\ref{app:counterexamples} (Construction~2)}, where three dimers
share a center fragment, every dimer residual vanishes, yet an order-3
cooperative move strictly improves the objective.

Proof of (c): \fzfix{for a cycle $C_L$ ($L\ge 3$) in the
fragment-intersection graph, place a binary variable $x_i$ on each fragment
of the cycle and let the edge $(i,i{+}1)$ carry the constraint
$x_{i+1}=x_i\oplus b_i$ with fixed bits satisfying
$\bigoplus_{i} b_i = 1$. Each dimer constraint is individually satisfiable,
so every independent dimer optimization can commit to a locally valid
boundary pair; but summing the constraints around the cycle gives
$0=\bigoplus_i b_i=1$, so no global assignment is consistent with all
committed pairs. The smallest case, $L=3$ with all $b_i=1$, is $G=K_3$ with
three singleton fragments and $q=2$: each edge admits a valid 2-coloring,
while $K_3$ is not 2-colorable.}
\end{proof}

\begin{corollary}\label{cor:treewidth}
If $\mathrm{treewidth}(I(G,\mathcal{F}))=1$ (forest), exact optimization
is achievable by order-2 message passing.
More generally, a junction-tree decomposition of~$I$ with bags of size
$\tau+1$ gives exact optimization \fzfix{at a cost proportional to the total
fragment state space of the largest bag, $\max_t \prod_{i\in B_t}
|\mathcal{X}_{F_i}|$, which is exponential in the number of original variables
covered by the bag, not merely in the graph-theoretic bag size~$\tau+1$.}
\end{corollary}

\begin{remark}\label{rem:observations}
This explains two phenomena visible in our experiments:
grid graphs certify at order~1: contiguous partitioning yields a
path-like (acyclic) intersection graph, \fzfix{and empirically the committed
Stage~1 cycle already reaches a globally optimal assignment on these
instances (the exactness theorem itself covers response-function message
passing, not committed cycling);}
dense random graphs have high delta values because their intersection
graphs contain many cycles, requiring higher-order corrections.
\end{remark}

% ---------------------------------------------------------------------------
\subsection{Relation to Junction Trees}

The acyclic message-passing result above is the disjoint-fragment analogue of
junction-tree dynamic programming. Junction trees usually use overlapping
bags rather than disjoint fragments, so they are not literally the same
object as the fragment partition used by the MBE heuristic. The connection is
that both methods become exact when the coordination graph is acyclic and the
messages retain response functions rather than committing prematurely to one
local assignment. \fzfix{The cost caveat applies equally to both: the message
tables range over full fragment-boundary states, so exact message passing is
exponential in the per-fragment state size, not merely in the number of
fragments or the treewidth of the fragment-intersection graph.}

\begin{proposition}\label{prop:jt}
Let $(T,\{B_t\}_{t\in V(T)})$ be a tree decomposition of a graph~$G$, and
let $f$ be a pairwise objective on~$G$. The standard junction-tree dynamic
program on the bags~$B_t$ computes a globally optimal assignment in time
exponential in the maximum bag size.
\end{proposition}

\begin{proof}
Assign each vertex and edge term of the pairwise objective to a bag that
contains its variables, which exists by the definition of tree decomposition.
For an oriented edge $t\to u$ of the tree decomposition, the message from
$t$ to $u$ is a function of the separator variables $B_t\cap B_u$ and is
obtained by minimizing over the variables in the component of
$T\setminus\{(t,u)\}$ containing~$t$, with the separator held fixed. The
running-intersection property ensures that these subproblems interact with
the rest of the graph only through the separator. A root minimization followed
by backtracking therefore recovers a globally optimal assignment.
\end{proof}

\begin{remark}
Thus junction-tree elimination is best viewed as the exact overlapping-bag
counterpart of the acyclic MBE message-passing case in \cref{thm:gap}(a).
The practical MBE algorithm studied here instead uses disjoint fragments,
black-box local solvers, and higher-order correction passes for non-tree
fragment interaction structures.
\end{remark}

% ===========================================================================
% ===========================================================================
\section{Computational Experiments}
\label{sec:experiments}

We test MBE on graph coloring ($q=3$--$15$) and MAX-CUT ($q=2$), with a
facility-location generality check in \cref{subsec:ufl}. All reported
objective values are conflict counts, so lower is better: monochromatic
edges for coloring, and uncut edges for MAX-CUT using the same
infrastructure with $q=2$. The primary comparisons use
TabuCol~\cite{hertz1987tabu}; solver-swap experiments replace the
fragment-level solver with SA or DSATUR while the pipeline stays fixed.
Tables report 10-seed averages, multi-start flat baselines, and equal-time
comparisons. In the \emph{sig} column, \emph{cert} means every seed reached
zero conflicts, which is globally optimal for these non-negative
objectives; the star markers ${}^{*}$, ${}^{**}$, and ${}^{***}$ report
two-sided paired $t$-tests at $p<0.05$, $p<0.01$, and $p<0.001$, where $p$
is the probability that seed-to-seed variation alone would produce a
difference at least this large.

For each comparison, the scientifically relevant control is not ``MBE
versus some other tuned solver'' but ``a flat solver versus the same solver
promoted by MBE under a comparable budget''. The tables should therefore be
read as wrapper-ablation tests: same instance, same initial assignment and
seeds, same local solver, and objective values always evaluated on the full
instance. The section runs from the controlled setting to the stress tests.
Multi-seed coloring, DIMACS, and MAX-CUT results establish the basic effect;
the sensitivity, equal-time, and multi-start subsections test whether it can
be explained away by parameter choice, CPU time, or restarts; the
solver-swap and partition experiments locate the boundary of the effect; the
facility-location study leaves the conflict-count setting altogether; and
the closing subsections reuse the same fragment data as a quantum resource
diagnostic.

% ---------------------------------------------------------------------------
\subsection{Implementation Structure}
\label{subsec:implementation}

All experiments run on a single Haskell implementation (built with GHC
via \texttt{cabal}; \cref{app:repro} records versions and hardware)
whose module structure follows the same separation as the paper. The
first layer, \texttt{Abstract.Configuration}, is the problem
specification: the set of assignments, the objective function, and the
operation that restricts attention to a chosen
fragment. \texttt{Problems.*} instantiates it per domain: for graph
coloring an assignment is a coloring and the objective counts
monochromatic edges; MAX-CUT reuses the same representation with two
colors and uncut edges as the objective; facility location supplies
the real-valued covering objective of \cref{subsec:ufl}. The second
layer, \texttt{Solvers.*}, is the restricted solver: it receives the
current global assignment, a set of variables it may change, and a
search budget, and returns a candidate assignment; TabuCol, SA, and
DSATUR all enter through this one interface. The third layer,
\texttt{Promotion.*}, is the MBE wrapper itself: the keep-if-not-worse
fragment sweeps (\texttt{GaussSeidel}), the delta-spectrum diagnostics
and dimer/trimer correction passes, and the merging of persistently
conflicting pairs (\texttt{Corrections},
\texttt{ManyBodyExpansion}). The benchmark drivers vary only the graph
family, color count, partition rule, seed, solver adapter, and
budget. Every table in this section is generated from seeded run logs
by the build scripts; the quantum diagnostics of
\cref{subsec:quantum} post-process exported fragment data with a
PennyLane-based simulation script.

% ---------------------------------------------------------------------------
\subsection{Graph Coloring: Multi-Seed Results}

We begin with a controlled graph-coloring comparison at $q=3$. This
experiment asks whether wrapping the same TabuCol routine in MBE changes the
quality of the final assignment. The instances include two grids, where a
zero-conflict 3-coloring is easy and serves as a sanity check, two queen
graphs with much denser local structure, one Mycielski graph, and a random
graph family whose size ranges from $50$ to $500$ vertices.

% Static copy of the pipeline-generated fragment; numbers verified
% identical to the seeded-log output by both authors. Regenerate via make.
\begin{table}[ht]
\centering
\caption{MBE vs.\ flat TabuCol, $q=3$ colors, 10 seeds.}
\label{tab:coloring}
\small
\begin{tabular}{@{}lrrr ccc c@{}}
\toprule
Instance & $N$ & $|E|$ & $k$
  & MBE (mean$\pm$std) & flat (mean$\pm$std) & $\Delta$ (mean$\pm$std) & sig \\
\midrule
grid-$5\times5$   & 25  & 40   & 5  & $0.0\pm0.0$   & $0.0\pm0.0$   & $+0.0\pm0.0$   & cert \\
grid-$8\times8$   & 64  & 112  & 8  & $0.0\pm0.0$   & $0.0\pm0.0$   & $+0.0\pm0.0$   & cert \\
queen-$5\times5$  & 25  & 160  & 5  & $30.0\pm0.8$  & $36.1\pm1.1$  & $+6.1\pm1.5$   & $^{***}$ \\
queen-$6\times6$  & 36  & 290  & 6  & $59.2\pm1.4$  & $70.6\pm2.2$  & $+11.4\pm2.1$  & $^{***}$ \\
myciel-5          & 23  & 71   & 5  & $4.0\pm0.0$   & $5.3\pm0.5$   & $+1.3\pm0.5$   & $^{***}$ \\
rand-50-0.1       & 50  & 132  & 8  & $4.1\pm0.3$   & $10.3\pm0.9$  & $+6.2\pm1.1$   & $^{***}$ \\
rand-100-0.1      & 100 & 488  & 12 & $60.8\pm3.6$  & $87.9\pm1.1$  & $+27.1\pm3.6$  & $^{***}$ \\
rand-200-0.06     & 200 & 1187 & 20 & $169.7\pm6.9$ & $224.4\pm1.3$ & $+54.7\pm6.6$  & $^{***}$ \\
rand-500-0.03     & 500 & 3719 & 39 & $577.3\pm9.3$ & $750.1\pm2.0$ & $+172.8\pm9.1$ & $^{***}$ \\
\bottomrule
\end{tabular}
\end{table}

Both methods find zero-conflict grid colorings, so those rows mainly check
that the wrapper preserves an already easy instance. The nontrivial cases
show a consistent reduction in conflicts. On the queen graphs, MBE reduces
the mean objective by $6.1$ and $11.4$ conflicts. On the random coloring
family, the absolute gain grows with instance size: $+6.2$ conflicts at
$n=50$, $+27.1$ at $n=100$, $+54.7$ at $n=200$, and $+172.8$ at $n=500$. The
standard deviations are small relative to these gains in the larger random
instances, so the improvement is consistent across seeds.

% ---------------------------------------------------------------------------
\subsection{Fragment-Count Sensitivity}

The fragment count $k$ is the main structural parameter in the wrapped
solver. It controls the compromise between local search freedom and boundary
coordination. With very few fragments, each local subproblem remains large,
and the solver receives little benefit from splitting. With too many
fragments, each subproblem becomes easy and many interactions move to
fragment boundaries, where higher-order passes revisit them.

The sensitivity table fixes the instance to rand-100-0.1 and varies only
$k$. The columns O1, O2, and O3 are the objectives after Stage~1
self-consistency, dimer (order-2) corrections, and trimer (order-3)
corrections, respectively. D1 and D2 are the total positive dimer delta
values after O1 and O2. The final $\Delta$ column again reports improvement
over the flat run.

% Static copy of the pipeline-generated fragment; numbers verified
% identical to the seeded-log output by both authors. Regenerate via make.
\begin{table}[ht]
\centering
\caption{Fragment count~$k$ sensitivity on rand-100-0.1. Flat TabuCol gives
87 conflicts in this run.}
\label{tab:ksens}
\small
\begin{tabular}{@{}r rrr rr r@{}}
\toprule
$k$ & O1 & O2 & O3 & D1 & D2 & $\Delta$ \\
\midrule
2  & 75 & 75 & 75 & 33 & 33 & $+12$ \\
5  & 77 & 77 & 77 & 55 & 55 & $+10$ \\
12 & 62 & 60 & 59 & 59 & 57 & $+28$ \\
25 & 66 & 66 & 66 & 64 & 64 & $+21$ \\
\bottomrule
\end{tabular}
\end{table}

The table shows the expected U-shaped sensitivity. In this instance the best
performance occurs near eight nodes per fragment.

On rand-100-0.1 with $k=12$, the progression is
init${}=89$, O1${}=62$, O2${}=60$, O3${}=59$, flat${}=87$.
The self-consistent first stage captures most of the gain, while dimer and
trimer corrections add smaller improvements. The delta totals also show how
the dimer stage acts: at $k=12$, the total dimer delta falls from 59 after
O1 to 57 after O2, while the objective improves from 62 to 60. Delta guides
the search toward remaining interactions; the reported quantity is the
accepted objective improvement.

% ---------------------------------------------------------------------------
\subsection{DIMACS Benchmark Instances}

The next comparison moves from synthetic 3-coloring examples to standard
DIMACS coloring instances. Here the number of colors is chosen per instance,
and the graphs include the DSJC random-graph family, a Leighton instance,
a queen graph, and myciel5. These are more representative of the benchmark
setting in which a flat coloring heuristic may already be quite strong.

% Static copy of the pipeline-generated fragment; numbers verified
% identical to the seeded-log output by both authors. Regenerate via make.
\begin{table}[ht]
\centering
\caption{DIMACS benchmarks with TabuCol, 10 seeds.}
\label{tab:dimacs}
\small
\begin{tabular}{@{}lrrrr ccc c@{}}
\toprule
Instance & $N$ & $|E|$ & $k$ & $q$
  & MBE (mean$\pm$std) & flat (mean$\pm$std) & $\Delta$ (mean$\pm$std) & sig \\
\midrule
DSJC125.1  & 125 & 736   & 12 & 5  & $10.8\pm1.2$  & $23.6\pm1.8$  & $+12.8\pm2.2$ & $^{***}$ \\
DSJC250.1  & 250 & 3218  & 25 & 8  & $36.6\pm2.6$  & $62.5\pm0.7$  & $+25.9\pm2.9$ & $^{***}$ \\
DSJC500.1  & 500 & 12458 & 39 & 12 & $102.2\pm4.7$ & $152.4\pm0.8$ & $+50.2\pm4.8$ & $^{***}$ \\
le450\_15a & 450 & 8168  & 39 & 15 & $35.8\pm2.3$  & $58.8\pm2.7$  & $+23.0\pm4.2$ & $^{***}$ \\
queen8\_8  & 64  & 728   & 8  & 9  & $5.5\pm1.6$   & $9.3\pm2.6$   & $+3.8\pm3.3$  & $^{**}$ \\
myciel5    & 47  & 236   & 5  & 6  & $0.0\pm0.0$   & $0.0\pm0.0$   & $+0.0\pm0.0$  & cert \\
\bottomrule
\end{tabular}
\end{table}

MBE improves TabuCol on every nontrivial DIMACS row. The largest absolute
gains occur on the larger random and Leighton instances: $+50.2$ conflicts
on DSJC500.1 and $+23.0$ on le450\_15a. The myciel5 row is another useful
sanity check: both methods reach zero conflicts, and MBE preserves that
solution. These rows measure the effect of the wrapper on the same
underlying TabuCol routine.

% ---------------------------------------------------------------------------
\FloatBarrier
\subsection{MAX-CUT ($q=2$)}

For MAX-CUT we take $q=2$ and report the number of uncut edges, so a value
of zero means that every edge is cut.

% Static copy of the pipeline-generated fragment; numbers verified
% identical to the seeded-log output by both authors. Regenerate via make.
\begin{table}[ht]
\centering
\caption{MAX-CUT, $q=2$, 10 seeds. Same pipeline as coloring.
{Reported values are uncut-edge counts (violations of the cut
objective); lower is better.}}
\label{tab:maxcut}
\small
\begin{tabular}{@{}lrrr ccc c@{}}
\toprule
Instance & $N$ & $|E|$ & $k$
  & MBE (mean$\pm$std) & flat (mean$\pm$std) & $\Delta$ (mean$\pm$std) & sig \\
\midrule
mc-grid-$5\times5$    & 25  & 40   & 5  & $0.0\pm0.0$    & $4.5\pm3.5$    & $+4.5\pm3.5$    & $^{**}$ \\
mc-grid-$8\times8$    & 64  & 112  & 8  & $0.0\pm0.0$    & $21.0\pm3.4$   & $+21.0\pm3.4$   & $^{***}$ \\
mc-rand-50-0.1        & 50  & 132  & 8  & $35.8\pm1.5$   & $42.9\pm0.3$   & $+7.1\pm1.6$    & $^{***}$ \\
mc-rand-100-0.1       & 100 & 488  & 12 & $157.1\pm2.3$  & $192.9\pm0.3$  & $+35.8\pm2.5$   & $^{***}$ \\
mc-rand-100-0.5       & 100 & 2494 & 12 & $1081.0\pm7.8$ & $1184.4\pm4.1$ & $+103.4\pm9.7$  & $^{***}$ \\
mc-rand-200-0.06      & 200 & 1187 & 20 & $392.9\pm8.6$  & $481.6\pm0.9$  & $+88.7\pm8.9$   & $^{***}$ \\
\bottomrule
\end{tabular}
\end{table}

The grid cases illustrate the convention: because the grids are bipartite,
a perfect cut exists, and MBE finds zero violations in all seeds. The random
instances show the same qualitative pattern as coloring. MBE reduces the
mean violation count by $35.8$ on mc-rand-100-0.1, by $103.4$ on the denser
mc-rand-100-0.5 instance, and by $88.7$ on mc-rand-200-0.06. The same solver
interface and objective representation carry across pairwise problems
unchanged.

% ---------------------------------------------------------------------------
\FloatBarrier
\subsection{Equal-Time Comparison}

The previous tables compare a wrapped run with a flat run under the standard
budgets used by the implementation. The equal-time comparison controls for
CPU time by giving the flat solver at least as much runtime as the MBE run.

% Static copy of the pipeline-generated fragment; numbers verified
% identical to the seeded-log output by both authors. Regenerate via make.
\begin{table}[ht]
\centering
\caption{Equal-time comparison: flat receives at least the same CPU time as
MBE. {Representative seed 2026; CPU times
(\texttt{System.CPUTime}) compare like-for-like runs on the same machine and
are hardware-dependent in absolute terms.}}
\label{tab:equaltime}
\small
\begin{tabular}{@{}lr ll r l@{}}
\toprule
Instance & $q$ & MBE (time) & flat (time) & $\Delta$ & note \\
\midrule
DSJC125.1         & 5  & 11 (0.12\,s)   & 25 (0.15\,s)   & $+14$  & \\
DSJC250.1         & 8  & 41 (0.38\,s)   & 62 (0.68\,s)   & $+21$  & \\
DSJC500.1         & 12 & 103 (2.89\,s)  & 153 (6.17\,s)  & $+50$  & MBE $2\times$ faster \\
le450\_15a        & 15 & 37 (1.18\,s)   & 56 (3.63\,s)   & $+19$  & MBE $3\times$ faster \\
rand-500-0.03     & 3  & 566 (1.84\,s)  & 752 (2.26\,s)  & $+186$ & \\
mc-rand-100-0.5   & 2  & 1076 (0.57\,s) & 1178 (0.79\,s) & $+102$ & \\
\bottomrule
\end{tabular}
\end{table}

MBE has the lower objective in every listed case, and on the two large
DIMACS instances it is faster outright: on DSJC500.1, MBE reaches 103
conflicts against 153 for flat TabuCol with the flat run taking roughly
twice the CPU time, and on le450\_15a it reaches 37 against 56 with the flat
run over three times slower. (Absolute timings in the table are
machine-dependent; \cref{app:repro} records the hardware, and the
MBE-to-flat time ratios are the reproducible quantity.) In these runs the
correction sequence improves both objective value and runtime.

% ---------------------------------------------------------------------------
\FloatBarrier
\subsection{Multi-Start Baseline}

Another baseline spends the flat solver's budget on several independent
starts. This compares MBE's structured order-based search with random basin
sampling. We split the same total flat budget across ten restarts and keep
the best result.

% Static copy of the pipeline-generated fragment; numbers verified
% identical to the seeded-log output by both authors. Regenerate via make.
\begin{table}[ht]
\centering
\caption{MBE vs.\ 10-start flat TabuCol. {Representative seed 2026.}
The flat budget is split across
restarts.}
\label{tab:multistart}
\small
\begin{tabular}{@{}lr rrr l@{}}
\toprule
Instance & $q$ & MBE & flat$\times$1 & flat$\times$10 & MBE advantage \\
\midrule
DSJC125.1       & 5  & 11   & 25   & 25   & $+14$ vs.\ flat$\times$10 \\
DSJC500.1       & 12 & 103  & 153  & 153  & $+50$ vs.\ flat$\times$10 \\
le450\_15a      & 15 & 37   & 56   & 56   & $+19$ vs.\ flat$\times$10 \\
rand-500-0.03   & 3  & 566  & 752  & 752  & $+186$ vs.\ flat$\times$10 \\
mc-rand-100-0.5 & 2  & 1076 & 1181 & 1205 & $+129$ vs.\ flat$\times$10 \\
\bottomrule
\end{tabular}
\end{table}

Under the fixed total budget used here, splitting the flat run across ten
random restarts gives little additional benefit. The wrapped solver remains
substantially better on the tested instances. On mc-rand-100-0.5, the
10-start flat run is even worse than the single-start flat run under this
budget split, and MBE improves over both. The gap therefore reflects the
structure of the MBE search, not merely access to more initial conditions.

% ---------------------------------------------------------------------------
\FloatBarrier
\subsection{Solver-Agnostic Demonstration}

The solver interface requires a restricted subproblem call and a candidate
assignment in return. To test this interface, we run the same order-based
wrapper with TabuCol, SA, and DSATUR. DSATUR is included deliberately as a
likely negative case: its ordering rule depends strongly on global
saturation information, which fragmentation removes.

% Static copy of the pipeline-generated fragment; numbers verified
% identical to the seeded-log output by both authors. Regenerate via make.
\begin{table}[ht]
\centering
\caption{Solver comparison under the same MBE pipeline, contiguous
partitioning, 10 seeds.}
\label{tab:agnostic}
\small
\begin{tabular}{@{}llr cc c c@{}}
\toprule
Instance & solver & $N$
  & MBE (mean$\pm$std) & flat (mean$\pm$std) & $\Delta$ (mean$\pm$std) & sig \\
\midrule
DSJC125.1       & TabuCol & 125 & $10.8\pm1.2$   & $23.6\pm1.8$   & $+12.8\pm2.2$  & $^{***}$ \\
DSJC125.1       & SA      & 125 & $10.8\pm2.2$   & $11.7\pm2.6$   & $+0.9\pm3.9$   &  \\
DSJC250.1       & TabuCol & 250 & $36.6\pm2.6$   & $62.5\pm0.7$   & $+25.9\pm2.9$  & $^{***}$ \\
DSJC250.1       & SA      & 250 & $41.8\pm3.4$   & $55.4\pm6.4$   & $+13.6\pm8.0$  & $^{***}$ \\
DSJC500.1       & TabuCol & 500 & $102.2\pm4.7$  & $152.4\pm0.8$  & $+50.2\pm4.8$  & $^{***}$ \\
DSJC500.1       & SA      & 500 & $111.6\pm8.4$  & $153.0\pm0.0$  & $+41.4\pm8.4$  & $^{***}$ \\
le450\_15a      & TabuCol & 450 & $35.8\pm2.3$   & $58.8\pm2.7$   & $+23.0\pm4.2$  & $^{***}$ \\
le450\_15a      & SA      & 450 & $37.2\pm3.1$   & $61.9\pm0.3$   & $+24.7\pm3.2$  & $^{***}$ \\
rand-500-0.03   & TabuCol & 500 & $577.3\pm9.3$  & $750.1\pm2.0$  & $+172.8\pm9.1$ & $^{***}$ \\
rand-500-0.03   & SA      & 500 & $553.1\pm5.9$  & $572.0\pm17.8$ & $+18.9\pm20.0$ & $^{*}$ \\
mc-rand-100-0.1 & TabuCol & 100 & $157.1\pm2.3$  & $192.9\pm0.3$  & $+35.8\pm2.5$  & $^{***}$ \\
mc-rand-100-0.1 & SA      & 100 & $145.9\pm3.4$  & $146.8\pm3.7$  & $+0.9\pm3.5$   &  \\
{rand-100-0.1}    & {DSATUR} & {100} & {$61.0\pm0.0$}  & {$58.8\pm0.6$}  & {$-2.2\pm0.6$}  & {${}^{***}$} \\
{rand-200-0.06}   & {DSATUR} & {200} & {$166.0\pm0.0$} & {$164.7\pm0.5$} & {$-1.3\pm0.5$}  & {${}^{***}$} \\
{rand-500-0.03}   & {DSATUR} & {500} & {$549.0\pm0.0$} & {$534.9\pm5.8$} & {$-14.1\pm5.8$} & {${}^{***}$} \\
{mc-rand-100-0.1} & {DSATUR} & {100} & {$152.0\pm0.0$} & {$160.0\pm2.0$} & {$+8.0\pm2.0$} & {${}^{***}$} \\
\bottomrule
\end{tabular}
\end{table}

TabuCol benefits most. SA improves on most rows, with two non-significant
cases. DSATUR is the measured negative control: on the coloring family,
promotion significantly \emph{hurts} it (\cref{tab:agnostic}, all
$p<0.001$), because DSATUR relies on global saturation-degree ordering and
fragmentation removes precisely the information that makes it effective. On
MAX-CUT, where ordering matters less, the promoted DSATUR is mildly positive
(for $q=2$ the same solver applies unchanged: saturation-degree-ordered
greedy reassignment over the two classes, with the uncut-edge count as the
objective). The pattern is a property of the wrapped solver: promotion helps
solvers whose strength is local search under fixed boundaries, and can harm
solvers whose strength is a global ordering heuristic.

% ---------------------------------------------------------------------------
\FloatBarrier
\subsection{Partition Sensitivity}
\label{sec:partition-sensitivity}

Standard graph-partitioning intuition suggests that a good decomposition
should minimize cross-fragment edges. For MBE this is a hypothesis to test
rather than assume: boundary edges are visible to the delta values and can be
revisited by dimer and higher-order passes, while the partition also controls
the search space available inside each fragment. We measure the latter by the
mean intra-fragment edge density
\[
  \bar\rho = \frac{1}{k}\sum_{a=1}^{k}
  \frac{2\,|E(F_a)|}{|F_a|(|F_a|-1)},
\]
with singleton fragments assigned density zero.

The two quantities are not independent. Every edge is either internal or cut,
so $\sum_a |E(F_a)| = |E| - \mathrm{cut}(\mathcal F)$: at fixed fragment
sizes, any reduction in the cut necessarily concentrates the same edges
inside the fragments. Cut and density are two views of a single trade-off,
not two independent partition criteria.

\begin{table}[ht]
\centering
\caption{Partition comparison for TabuCol, 10 seeds. The BFS-grown bisection reduces the cross-edge count and raises the mean intra-fragment density $\bar\rho$; fragments are held approximately balanced. Cut is the cross-edge count (contiguous/bisection). The MBE gain is statistically indistinguishable between the two partitions despite the four- to eight-fold density difference.}\label{tab:partition}
\small
\setlength{\tabcolsep}{4pt}
\begin{tabular}{@{}lrr r cc cc@{}}
\toprule
Instance & $q$ & $k$ & cut (c/b)
  & contig $\Delta$ & bisect $\Delta$
  & $\bar\rho_{\mathrm{c}}$ & $\bar\rho_{\mathrm{b}}$ \\
\midrule
rand-100-0.1     & 3 & 12 & 454/399   & $+27.1\pm3.6$  & $+31.3\pm2.4$  & $0.079$ & $0.308$ \\
rand-200-0.06    & 3 & 20 & 1150/1042 & $+54.7\pm6.6$  & $+53.7\pm6.9$  & $0.041$ & $0.260$ \\
rand-500-0.03    & 3 & 39 & 3640/3429 & $+172.8\pm9.1$ & $+175.9\pm16.9$& $0.026$ & $0.215$ \\
mc-rand-100-0.1  & 2 & 12 & 454/399   & $+35.8\pm2.5$  & $+37.2\pm3.1$  & $0.079$ & $0.308$ \\
mc-rand-200-0.06 & 2 & 20 & 1150/1042 & $+88.7\pm8.9$  & $+93.0\pm7.2$  & $0.041$ & $0.260$ \\
\bottomrule
\end{tabular}
\end{table}

Table~\ref{tab:partition} compares the contiguous partition with a BFS-grown
bisection that reduces cross edges while holding fragments approximately
balanced (Appendix~\ref{app:repro}). Two facts stand out. First, on
unstructured random graphs the cut reduction available to a graph-aware
partitioner is modest, $6$--$12\%$; this is expected, since $G(n,p)$ has no
community structure to exploit. Second, the same partitions differ sharply in
internal density: bisection raises $\bar\rho$ by factors of four to eight
(from $0.026$ to $0.215$ on rand-500-0.03), because BFS-grown fragments are
locally connected whereas contiguous index blocks are essentially random
subsets at the ambient density.

Despite that four- to eight-fold change in the local search space, the
promoted solver's advantage is unchanged: the MBE gain under bisection is
statistically indistinguishable from, or marginally better than, the gain
under contiguous partitioning on every coloring and MAX-CUT instance tested.
Within the range probed here, neither the cut nor the intra-fragment density
predicts the size of the MBE gain.

The practical reading is that on unstructured instances the promotion
mechanism is robust to the partition: a naive contiguous split performs as
well as a graph-aware one, and no partitioning package is needed to obtain
the reported gains. What the partition must not do is degenerate --- a
partitioner that lets one fragment absorb most of the graph reduces MBE to a
flat solve. This robustness is a property of these instances rather than a
universal claim: the facility-location study of Section~\ref{subsec:ufl}
exhibits the opposite behavior and identifies the condition that separates
the two cases.

% ---------------------------------------------------------------------------
\subsection{Generality Check: Facility Location}
\label{subsec:ufl}

To confirm that the framework is not specific to coloring-type
objectives, we also apply it to uncapacitated facility location (UFL):
minimize $\sum_{i\in\mathrm{Open}} f_i + \sum_c \min_{i\in\mathrm{Open}}
\alpha\, d(c,i)$ over the set of open facilities, on Euclidean instances
(facilities and clients uniform in the unit square, opening costs
$f_i\sim U[1.0,2.5]$, $\alpha=1$). Facilities are the sites, the open set is
the configuration, and the pluggable solver is Metropolis open/close flip
search; each client belongs to the fragment of its geometrically nearest
facility.
Because these objectives are real-valued, this experiment should
be interpreted as a generality check for the black-box promotion interface,
not as an application of the integer-valued convergence bound in
Theorem~\ref{thm:convergence}: monotonicity still holds, but the sharp
$M=f(x^{(0)})$ descent count does not apply to arbitrary real energies.
UFL also exposes a structural contrast with coloring: the fragment
interaction is \emph{cooperative} rather than repulsive. For covering-type
objectives the union of two fragments can serve boundary clients from the
cheaper side, so $E(A\cup B)-E(A)-E(B)\le 0$; the correction hierarchy is
unchanged, but the dimer pass selects the \emph{most negative} deltas
(largest unexploited sharing) instead of positive ones (frustration).
Unlike the pairwise conflict objectives above, UFL is used here as an
operational black-box promotion test: client assignment induces higher-order
dependence over open facilities, so the dimer spectrum is a heuristic
prioritization signal, not an exact second-order closure claim.
Table~\ref{tab:ufl} reports the same ten-seed protocol as coloring under two
partitions: geographically coherent fragments (facilities grouped into
spatial strips) versus generation-order fragments (spatially random).

% Static copy of the pipeline-generated fragment; numbers verified
% identical to the seeded-log output by both authors. Regenerate via make.
\begin{table}[ht]
\centering
\caption{{Facility location (UFL), cost (lower is better);
$\Delta=E_{\text{flat}}-E_{\text{MBE}}$, paired-$t$ markers as elsewhere.
{One fixed generated instance per size (generator seed $42$); the ten
seeds vary the initial open set and solver randomness. On these instances the
geographic-partition gain is significant under solver-seed variation and
increases with size; under spatially random partitioning it vanishes or
reverses. $E^\ast$ is the proven optimum (GLPK mixed-integer solve of the
standard UFL formulation); gap $=(E_{\text{MBE}}-E^\ast)/E^\ast$.
$E^\ast$ is instance-level and shown once, on the geographic rows.}}}
\label{tab:ufl}
\footnotesize
\setlength{\tabcolsep}{2.5pt}
\begin{tabular}{@{}lrrr ccc l cc@{}}
\toprule
{Instance} & {$n_f$} & {$n_c$} & {$k$}
  & {$E_{\text{MBE}}$} & {$E_{\text{flat}}$} & {$\Delta$} & {sig}
  & {$E^\ast$} & {gap} \\
\midrule
{\emph{geographic partition}} & & & & & & & & & \\
{ufl-25-100}  & {25}  & {100} & {5}  & {$24.24\pm0.16$} & {$24.69\pm0.41$} & {$+0.44\pm0.37$ ($+1.8\%$)} & {${}^{**}$}  & {$24.11$} & {$0.5\%$} \\
{ufl-50-200}  & {50}  & {200} & {8}  & {$39.98\pm0.28$} & {$40.42\pm0.43$} & {$+0.44\pm0.54$ ($+1.1\%$)} & {${}^{*}$}   & {$39.48$} & {$1.3\%$} \\
{ufl-100-400} & {100} & {400} & {12} & {$63.04\pm0.11$} & {$64.08\pm0.51$} & {$+1.04\pm0.58$ ($+1.6\%$)} & {${}^{***}$} & {$62.86$} & {$0.3\%$} \\
{ufl-200-800} & {200} & {800} & {16} & {$94.83\pm0.50$} & {$97.38\pm1.12$} & {$+2.55\pm1.26$ ($+2.6\%$)} & {${}^{***}$} & {$93.55$} & {$1.4\%$} \\
\addlinespace
{\emph{spatially random partition}} & & & & & & & & & \\
{ufl-25-100}  & {25}  & {100} & {5}  & {$24.46\pm0.25$} & {$24.69\pm0.41$} & {$+0.23\pm0.49$ ($+0.9\%$)} & {}          & & \\
{ufl-50-200}  & {50}  & {200} & {8}  & {$40.86\pm0.62$} & {$40.42\pm0.43$} & {$-0.43\pm0.72$ ($-1.1\%$)} & {}          & & \\
{ufl-100-400} & {100} & {400} & {12} & {$64.75\pm0.95$} & {$64.08\pm0.51$} & {$-0.66\pm0.83$ ($-1.0\%$)} & {${}^{*}$}  & & \\
{ufl-200-800} & {200} & {800} & {16} & {$97.89\pm1.60$} & {$97.38\pm1.12$} & {$-0.51\pm1.46$ ($-0.5\%$)} & {}          & & \\
\bottomrule
\end{tabular}
\end{table}

The UFL experiment is a black-box promotion and generality
  check, not an equal-compute comparison: the MBE pipeline consumes up
  to roughly nine times the flat solver's flip budget at
  $n_f=200$. The certified-optimum column bounds what any further
  computation could recover.  Two observations carry over from
  coloring. First, the promoted solver is not only better on average
  but far more reliable: at $n_f=100$ the MBE standard deviation is
  $0.11$ against the flat solver's $0.51$, and the promoted
    solver sits within $0.3$--$1.4\%$ of the proven optimum at every
    size ($E^\ast$ column, GLPK-certified) versus $2.0$--$4.1\%$ for
    the flat solver. Second, the partition must match the objective's
  interaction locality. This is where the covering objective departs from the
  coloring benchmarks, whose MBE gain is insensitive to the partition
  (Section~\ref{sec:partition-sensitivity}):
  geographically coherent fragments keep client territories intact and
  MBE wins at every size, while spatially random fragments scatter
  each territory across the system and the same pipeline becomes
  neutral or harmful. The same structure dependence appears in a further non-coloring domain: an auxiliary Ising-model check promotes a spin-flip
solver on one-dimensional ferromagnetic chains, whose nearest-neighbor
interactions any contiguous partition respects, and on fully connected
Sherrington--Kirkpatrick couplings~\cite{sherrington1975solvable}, which no
partition can localize. The chain and SK cases bracket the same
locality-of-interaction requirement seen above; these auxiliary results are
reproducible from the accompanying code (\texttt{--ising-benchmark}) but
not tabulated here.

% ---------------------------------------------------------------------------
\subsection{Quantum Resource Diagnostics}
\label{subsec:quantum}

The final experiments use the same fragment information for a resource
question rather than a new classical benchmark: if an expensive quantum
routine is available only on small active regions, which regions should be
sent to it? The implementation exports the fragments, boundaries, and
positive-delta groups produced by the classical MBE run; a resource script
then computes three quantities from the exported active regions: the number
of qubits required by different encodings, a small depth-2 QAOA calibration
as a function of active qubit count, and the lightcone sizes relevant to
simulation. These subsections are resource diagnostics and an outlook, not
performance claims; the simulation protocol is recorded in
\cref{app:repro}. The allocation rule is the delta spectrum itself:
zero-delta fragments are left classical, and positive-delta neighborhoods
define candidate subproblems for quantum hardware.

% Static copy of the pipeline-generated fragment; numbers verified
% identical to the seeded-log output by both authors. Regenerate via make.
\begin{table}[ht]
\centering
\caption{Qubit requirements at scale for graph coloring with 3 colors.
$Q_\text{pen}$: penalty QUBO ($n\times q$).
$Q_\text{pro}$: profit encoding ($n$).
$Q_\text{hyb}$: MBE + profit, using the largest frustrated fragment plus
boundary.}
\label{tab:qubits}
\small
\begin{tabular}{@{}rr r r rrr r r@{}}
\toprule
$n$ & $|E|$ & $k$ & zero-$\Delta$
  & $Q_\text{pen}$ & $Q_\text{pro}$ & $Q_\text{hyb}$ & ratio & MBE+ \\
\midrule
16  & 24   & 4  & 4/4  & 48   & 16  & 0   & $\infty$ & $+0$   \\
50  & 132  & 8  & 2/8  & 150  & 50  & 35  & $4\times$ & $+4$   \\
100 & 488  & 12 & 0/12 & 300  & 100 & 72  & $4\times$ & $+27$  \\
200 & 1187 & 20 & 0/20 & 600  & 200 & 110 & $5\times$ & $+60$  \\
500 & 3719 & 39 & 0/39 & 1500 & 500 & 188 & $8\times$ & $+186$ \\
\bottomrule
\end{tabular}
\end{table}

\Cref{tab:qubits} reports the resulting qubit counts for 3-coloring. The
penalty QUBO column uses $nq$ binary variables, while the profit encoding
uses $n$. MBE plus profit encoding further restricts the quantum call to the
largest frustrated fragment plus its boundary. The resulting hybrid qubit
count is 72 for $n=100$, 110 for $n=200$, and 188 for $n=500$, corresponding
to reductions of $4\times$, $5\times$, and $8\times$ relative to the penalty
encoding.

% ---------------------------------------------------------------------------
\subsection{QAOA Approximation Quality}

Qubit counts alone do not say whether a quantum call is useful. We therefore
include a small calibration experiment for the profit-encoded depth-2 QAOA
routine used by the resource script. The calibration groups random active
subproblems by qubit count and compares the QAOA value with the exact
optimum for those small subproblems.

% Static copy of the pipeline-generated fragment; numbers verified
% identical to the seeded-log output by both authors. Regenerate via make.
\begin{table}[ht]
\centering
\caption{Depth-2 profit QAOA approximation ratio
(5 random graphs per size, density~0.35).}
\label{tab:qaoa}
\small
\begin{tabular}{@{}r cc cc l@{}}
\toprule
Qubits & Avg ratio & Min ratio & Avg opt & Avg QAOA & Assessment \\
\midrule
4--8   & 1.00 & 1.00 & 3.3 & 3.3 & Optimal  \\
9--10  & 0.95 & 0.67 & 4.5 & 4.2 & Good     \\
11--12 & 0.83 & 0.67 & 4.6 & 3.8 & Degraded \\
\bottomrule
\end{tabular}
\end{table}

Depth-2 profit QAOA is reliable on the smallest subproblems and degrades
once the active region reaches roughly 9--12 qubits. MBE is useful here not
because it establishes quantum advantage, but because it identifies smaller
active neighborhoods where future quantum calls can be concentrated.

% ---------------------------------------------------------------------------
\subsection{Lightcone Reduction}

The same exported active regions also determine the \emph{lightcone} size
for QAOA simulation. At depth~$p$, evaluating a local expectation value
depends only on variables within a radius-$p$ neighborhood, i.e., the
lightcone, of the interaction graph. On the full graph this lightcone can cover
most of a large random instance. After MBE, the candidate quantum call is
restricted to a fragment-local active region plus its boundary, so the
corresponding lightcone can be much smaller.

% Static copy of the pipeline-generated fragment; numbers verified
% identical to the seeded-log output by both authors. Regenerate via make.
\begin{table}[ht]
\centering
\caption{Full-graph QAOA lightcone sizes for one expectation-value
evaluation at depths $p=1$ and $p=2$.}
\label{tab:full-lightcone}
\small
\begin{tabular}{@{}l rr r rr@{}}
\toprule
Instance & $N$ & $|E|$ & avg.\ degree & $p=1$ & $p=2$ \\
\midrule
grid-$4\times4$   & 16  & 24   & 3.0  & 8  & 14  \\
queen-$5\times5$  & 25  & 160  & 12.8 & 23 & 25  \\
rand-50-0.1       & 50  & 132  & 5.3  & 19 & 48  \\
rand-100-0.1      & 100 & 488  & 9.8  & 34 & 99  \\
rand-200-0.06     & 200 & 1187 & 11.9 & 42 & 185 \\
rand-500-0.03     & 500 & 3719 & 14.9 & 50 & 399 \\
\bottomrule
\end{tabular}
\end{table}

% Static copy of the pipeline-generated fragment; numbers verified
% identical to the seeded-log output by both authors. Regenerate via make.
\begin{table}[ht]
\centering
\caption{Lightcone sizes: full graph vs.\ MBE fragments ($p=2$).}
\label{tab:lightcone}
\small
\begin{tabular}{@{}l rr r l@{}}
\toprule
Instance & full $p{=}2$ & frag $p{=}2$ & saved & note \\
\midrule
rand-50-0.1    & 48  & 35  & $+13$  & \\
rand-100-0.1   & 99  & 71  & $+28$  & \\
rand-200-0.06  & 185 & 102 & $+83$  & \\
rand-500-0.03  & 399 & 157 & $+242$ & full intractable \\
\bottomrule
\end{tabular}
\end{table}

On rand-500, the full $p=2$ lightcone is 399 qubits, while the largest
MBE-fragment lightcone is 157 qubits; across the random family, the saving
grows from 13 qubits at $n=50$ to 242 at $n=500$. This does not make
depth-2 QAOA competitive with classical heuristics at present; it makes the
resource-allocation problem explicit and substantially smaller.

% ===========================================================================
\section{Related Work}

\paragraph{Many-body and fragmentation methods.}
Many-Body Expansion methods originate in computational chemistry, where the
energy of a large molecular system is decomposed into fragment, pair, triple,
and higher-order contributions~\cite{kitaura1999fmo,gordon2012fragmentation,
fedorov2007fmo}. The present work transfers this organizing principle to
discrete optimization. The common mathematical ingredient is M\"obius
inversion on a lattice of fragment subsets~\cite{rota1964mobius}. In our
setting, M\"obius inversion gives an exact evaluation expansion of the
objective: an objective whose terms involve at most $r$ fragments is recovered
at order~$r$. The optimization algorithm then uses this expansion as a
correction hierarchy around black-box local solver calls.

\paragraph{Decomposition in optimization.}
Classical decomposition methods also solve large problems by coordinating
smaller subproblems. ADMM~\cite{boyd2011admm} coordinates local solves through
dual variables and residual updates, and is especially effective for
continuous convex problems. When adapted to discrete optimization, however,
it does not generally inherit finite convergence guarantees. Dantzig-Wolfe
decomposition~\cite{dantzig1960decomposition} coordinates subproblems through
a master problem and column generation. MBE differs from both in that the
local solver is treated as an unchanged black box, while coordination is
handled by self-consistent cycling and higher-order fragment corrections.

\fzrev{\paragraph{Heuristic neighborhood and block methods.}
The Stage~1 cycling is, in isolation, block-coordinate local
search~\cite{tseng2001convergence}, and jointly re-solving fragment unions is
reminiscent of large-neighborhood search~\cite{shaw1998lns,pisinger2010lns},
ruin-and-recreate~\cite{schrimpf2000ruin}, and variable-neighborhood
search~\cite{mladenovic1997vns}, which also escape local minima by exploring
larger, structured neighborhoods. What MBE adds is not the neighborhoods but
the accounting: the M\"obius/evaluation hierarchy specifies which unions are
worth solving and in what order, the delta spectrum and dimer residuals give
per-pair diagnostics and certificates rather than a single acceptance test,
promotion is monotone for any black-box solver by construction, and truncation
carries the exactness statements of Section~3. In LNS terms, MBE supplies a
principled, problem-structure-driven neighborhood schedule with quantitative
stopping evidence, rather than a randomized destroy/repair heuristic.}

Junction-tree methods~\cite{lauritzen1988local} give exact dynamic
programs when the interaction structure has bounded treewidth. Our
acyclic message-passing result is the disjoint-fragment analogue of
this principle: when fragment interactions form a forest and messages
retain full response functions, exact optimization is possible. The
practical MBE algorithm studied here uses disjoint fragments and
heuristic correction passes, while junction trees use overlapping bags
and exact separator messages.  \fzadd{The distinction is sharpest when
  the fragment-intersection structure contains a cycle.} \fzfix{On a
  grid of $K_4$ cliques, for example, exact coordination on the same
  partition used by MBE (eliminating each fragment together with its
  boundary) requires bags of size $20$, i.e.~$>10^{12}$ table entries
  for $q=4$ (an optimal tree decomposition can reduce the bag size,
  but any decomposition must still span whole cliques and their
  matchings), whereas MBE's cycling-plus-dimer coordination on the
  identical partition costs $O(n\,B\,k)$, where $n$ is the number of
  sites, $B$ the per-fragment solver iteration budget
  (Appendix~\ref{app:repro}), and $k$ the fragment count
  ($\sim\!2\times10^5$ operations in our case). MBE reaches a
  zero-conflict coloring in our runs; exact elimination would
  necessarily find one too (the instances are $q$-colorable), at that
  cost.}  \fzadd{The comparison is by cost regime, not dominance:
  junction trees retain an exactness guarantee whenever the bag size
  is manageable; MBE trades that worst-case guarantee for a cheaper
  correction hierarchy whose quality is empirical and
  instance-dependent.}

\paragraph{Graph coloring heuristics.}
Graph coloring has a long history of effective local-search methods,
including TabuCol~\cite{hertz1987tabu}, reactive tabu search
\cite{blochliger2008reactive}, and hybrid evolutionary approaches
\cite{galinier1999hybrid}. These specialized solvers can reach stronger
absolute results than the simple local solvers used in our experiments. Our
goal is different: MBE is a promotion framework, not a new coloring heuristic.
The experiments ask whether the same solver improves when wrapped in the MBE
decomposition, and whether the same wrapper can be reused across structurally
different solvers.

\paragraph{Quantum-classical optimization.}
The quantum-resource experiments are motivated by the gap between the size of
interesting optimization instances and the size of subproblems that near-term
quantum routines can plausibly handle. Penalty-free profit encodings reduce
the number of qubits required for coloring-style objectives~\cite{angara2025penalty}.
QAOA~\cite{farhi2014qaoa} and tensor-network simulation tools such as
QTensor~\cite{lykov2021qtensor} remain strongly limited by active lightcone
size. In this context, the MBE delta spectrum is useful as a diagnostic for
where quantum resources might be concentrated. We do not claim quantum
advantage; the contribution is a decomposition and resource-allocation
mechanism that exposes smaller active subproblems.

% ===========================================================================
\section{Conclusion}

We presented a Many-Body Expansion framework for scaling local solvers for
discrete optimization. The framework partitions a problem into fragments,
\fzfix{cycles a black-box solver over those fragments until a sweep produces
no strict objective decrease,} and then applies higher-order corrections on
selected unions of fragments. The solver interface is deliberately minimal: the solver
only needs to accept a restricted subproblem and return a candidate assignment.

The formal results separate evaluation from optimization. M\"obius inversion
gives an exact fragment expansion of the objective, and pairwise objectives
such as graph coloring, MAX-CUT, and QUBO close at second order for
evaluation. Optimization still requires coordination among fragments. \fzfix{Stage~1
terminates finitely for non-negative integer-valued objectives under
non-worsening acceptance with strict-progress termination,} and the
correction hierarchy is monotone
because accepted updates never increase the full objective. For acyclic
fragment interaction structures, exact coordination is obtained by retaining
response-function messages, matching the dynamic-programming principle behind
junction-tree methods.

The experiments show that MBE can improve local search in practice. With the
same pipeline, TabuCol and simulated annealing generally produce better
solutions than their flat counterparts on graph coloring and MAX-CUT
benchmarks, often in less CPU time. The partition experiments show that on unstructured graphs
the MBE gain is robust to the choice of decomposition: graph-aware bisection
changes the mean intra-fragment density by factors of four to eight without
changing the promoted solver's advantage. Partition choice becomes decisive
only when the objective carries exploitable interaction structure, as in the
facility-location study, where fragments that ignore geometry erase the gain
entirely.

The delta spectrum should be interpreted as a diagnostic, not as a global
optimality certificate. It identifies fragment boundaries and fragment unions
where residual interaction remains large, and it is useful for prioritizing
classical corrections or quantum subproblem calls. Future work includes
wrapping stronger coloring heuristics, characterizing which objectives are
partition-sensitive and which are not, and integrating exact or quantum
subproblem solvers into the same MBE correction hierarchy.
\fzadd{A useful next benchmark is therefore not simply a larger graph, but a
three-way stress test: low treewidth instances where exact junction-tree
methods are competitive, dense instances where fragmentation is expected to
struggle, and intermediate structured instances where MBE should have its
largest practical advantage.}

\appendix
\section{Counterexample Constructions}
\label{app:counterexamples}

This appendix records the two explicit constructions referenced in
\cref{sec:mbe}. Construction~1 is cited in the remark following
\cref{prop:locality}: it separates the delta spectrum (a diagnostic) from
the dimer residual (a certificate) by exhibiting a state with zero delta
spectrum but positive residual. Construction~2 supports
\cref{thm:gap}\textup{(b)}: it shows that independently committed dimer
re-optimizations are not exact even when the fragment-intersection graph is
a tree, which is why exactness in \cref{thm:gap}\textup{(a)} requires
message passing with full response functions rather than committed local
moves.

\fzrev{Both constructions use explicit non-negative integer-weighted pairwise
and unary terms; unary terms can be realized as edges to auxiliary fixed
structure without changing the fragment-intersection topology.

\textbf{Construction 1 (zero delta spectrum, positive dimer residual).}
Two binary variables $a,b\in\{0,1\}$ in singleton fragments $F_1=\{a\}$,
$F_2=\{b\}$, with objective
\[
  f(a,b) \;=\; 5\,[a\neq b] \;+\; [a=1] \;+\; 3\,[b=0].
\]
At $x=(0,0)$: $f(x)=3$ and the cross term vanishes, so
$\Delta(F_1,F_2)=5\,[a\neq b]=0$. Both single-fragment re-optimizations are stuck (flipping
$b$ alone costs $5-3=2$ more; flipping $a$ alone costs $6$ more), yet the
joint dimer state $(1,1)$ has $f=1$, so
$g_{F_1\cup F_2}(x)=3-1=2>0$: the intra-fragment improvement of $b$'s unary
term is unlocked only by simultaneously moving the boundary variable $a$ in
the other fragment. Hence $\Delta_{\mathrm{total}}=0$ does not certify
$g_U=0$.

\textbf{Construction 2 (star: all dimer residuals vanish, order 3 improves).}
Binary variables $c,u_1,u_2,u_3$ in singleton fragments; cross terms couple
only the center to the leaves, so the fragment-intersection graph is the star
$K_{1,3}$ (a tree). Objective:
\[
  f(c,u_1,u_2,u_3) \;=\; 2\,[c=0] \;+\; \sum_{i=1}^{3}\,[c\neq u_i].
\]
At $x=(0,0,0,0)$: $f(x)=2$. Every single-fragment re-optimization is stuck. For each center--leaf
dimer $U=\{c,u_i\}$ the restricted objective satisfies
$A_U(0,0)=2$, $A_U(0,1)=3$, $A_U(1,0)=3$, $A_U(1,1)=2$
(the last from the two exterior disagreements $[c\neq u_j]$, $j\neq i$),
so $\min A_U = A_U(x_U)$ and $g_U(x)=0$; leaf--leaf dimers are uncoupled and
likewise have zero residual. Yet the order-3 move on $\{c,u_1,u_2\}$ to
$(1,1,1)$ with $u_3=0$ fixed gives $f=1<2$, a strict improvement, and the
full flip $(1,1,1,1)$ attains the optimum $f=0$. Committed independent dimer
re-optimization therefore cannot be exact even on trees: a fragment shared by
several dimers can require a cooperative move invisible to every individual
dimer residual, exactly as claimed in Theorem~\ref{thm:gap}(b).}

% ---------------------------------------------------------------------------
\section{Experimental Protocol}
\label{app:repro}

\fzadd{Every comparison in this paper isolates the MBE wrapper from solver
tuning: the flat and promoted runs share the instance, the initial assignment,
and the local-solver family, and all objective values are evaluated on the full
problem.} \fzdata{The accompanying source code regenerates the experiment logs behind
every non-exact-solver column (\texttt{scripts/regenerate\_paper\_data.py});
the certified $E^\ast$ column of the facility-location table is produced
separately by \texttt{scripts/ufl\_optimum\_glpk.py}, whose GLPK solution
certificates ship in \texttt{paper\_runs/ufl-optima/}. This
appendix records the exact settings used.}

\fzdata{\textbf{Seeds and statistics.} All multi-seed entries use the ten fixed
seeds $\{2026, 42, 77, 99, 55, 33, 1001, 2025, 314, 777\}$ and report
mean$\pm$standard deviation. Significance markers are paired $t$-tests on the
ten paired seed results.} \fzfix{Reported standard deviations follow the
population convention (divide by $n$); every significance marker in the paper
is identical under the population and sample conventions. Single-run entries
(Tables~\ref{tab:equaltime} and the multi-start table) use seed $2026$.}

\textbf{Hardware and software environment.} All experiment logs, including
the CPU times in Table~\ref{tab:equaltime}, were produced in a single run on
a MacBook Pro (Apple M2, 16\,GB RAM) under macOS, single-threaded, with the
Haskell code compiled by GHC~9.4.8 via \texttt{cabal}. CPU times are measured
with \texttt{System.CPUTime} and are hardware-dependent in absolute terms;
the MBE/flat time \emph{ratios} are the reproducible quantity. The $E^\ast$
certificates were produced by GLPK's \texttt{glpsol}.

\fzdata{\textbf{Initialization and partition.} Colorings are initialized by
descending-degree greedy assignment; the contiguous partition splits the site
list into $k$ roughly equal contiguous blocks, and the bisection partition is
BFS-grown to \fzfix{reduce} cross-edges, with the BFS frontier truncated at
the half-size limit so that fragments remain approximately balanced (an
unbalanced partitioner that lets one fragment absorb most of the graph makes
the promoted solver degenerate to a flat solve). Mean intra-fragment density is
$\bar\rho=\frac1k\sum_a 2|E(F_a)|/(|F_a|(|F_a|-1))$, with singleton fragments
assigned zero.}

\fzdata{\textbf{Local solvers and budgets.} The flat baseline receives budget
$15n$ for $n$ sites. Promoted runs use Stage~1 Gauss--Seidel cycling
(up to $15$ sweeps, $100$ solver iterations per fragment), a dimer correction
pass ($200$ iterations), and a trimer pass ($300$ iterations, fragment-union
size capped at $50$). TabuCol uses tenure $\max(1,\lfloor|F_a|/6\rfloor)$;
simulated annealing uses geometric cooling from $\max(1,\tfrac12 f)$ to
$0.01$; DSATUR uses saturation-degree ordering with a stagnation cap of $20$.}

\fzdom{\textbf{Facility location (Table~\ref{tab:ufl}).} Euclidean instances
with $n_f$ facilities and $n_c$ clients uniform in the unit square, opening
costs $f_i\sim U[1.0,2.5]$, assignment cost $\alpha\,d(c,i)$ with $\alpha=1$;
one fixed instance per size (generator seed $42$), with the ten run seeds
controlling the initial open set (each facility open with probability $0.3$)
and the solver randomness. The pluggable solver is Metropolis open/close flip
search with geometric cooling from $\max(0.05, 0.05\,E_0)$ to $10^{-3}$ and
best-so-far tracking; the flat baseline receives budget $15 n_f$. Fragment
energies use the territory convention (each client belongs to the fragment of
its geometrically nearest facility; a fragment with no open facility pays a
per-client penalty of $2\sqrt{2}\,\alpha$), under which the true objective is
recovered on the full site set and cross-fragment deltas are cooperative
($\le 0$). The MBE pipeline is Gauss--Seidel cycling ($100$ iterations per
fragment, up to $15$ sweeps) followed by a dimer pass over the $k$ most
negative cross-energy pairs ($200$ iterations each), guarded on the full
objective; \texttt{Promotion.*} is unchanged, only the pair-selection sign
differs. The geographic partition sorts facilities into
\fzfix{$\mathrm{round}(\sqrt{k})$} vertical strips by $x$, then by~$y$ within
each strip; the spatially random partition chops generation order.}

\fzfix{\textbf{Quantum diagnostics (Tables in Sections~4.10--4.12).} QAOA
results use depth $p=2$ circuits simulated with the PennyLane statevector
backend (feasible up to $20$ qubits), $50$ optimizer steps, and $2$ restarts
per instance; approximation ratios are averaged over $5$ random graphs per
size with edge density $0.35$ and per-graph seeds $71s+13n$ for
$s\in\{0,\dots,4\}$. \fzfix{Optimization uses PennyLane's
\texttt{GradientDescentOptimizer} with learning rate $0.08$; candidate
bitstrings are drawn as $1024$ samples of the output distribution, accepted
above a probability threshold of $0.003$, and coloring candidates are
completed greedily.} Qubit counts use the penalty encoding
($n\times q$ qubits) and the profit encoding ($n$ qubits); the hybrid column
takes the largest positive-delta fragment plus its boundary. The pipeline is
\texttt{--quantum-classical} followed by \texttt{paper\_qaoa\_results.py}.}

% ===========================================================================
\section*{Code and Data Availability}
The complete source code, the seeded experiment logs and table-generation
pipeline, and the GLPK optimality certificates accompany this submission;
every table in the paper regenerates from them.

\bibliography{cites}


\begin{thebibliography}{19}
\ifx \bisbn   \undefined \def \bisbn  #1{ISBN #1}\fi
\ifx \binits  \undefined \def \binits#1{#1}\fi
\ifx \bauthor  \undefined \def \bauthor#1{#1}\fi
\ifx \batitle  \undefined \def \batitle#1{#1}\fi
\ifx \bjtitle  \undefined \def \bjtitle#1{#1}\fi
\ifx \bvolume  \undefined \def \bvolume#1{\textbf{#1}}\fi
\ifx \byear  \undefined \def \byear#1{#1}\fi
\ifx \bissue  \undefined \def \bissue#1{#1}\fi
\ifx \bfpage  \undefined \def \bfpage#1{#1}\fi
\ifx \blpage  \undefined \def \blpage #1{#1}\fi
\ifx \burl  \undefined \def \burl#1{\textsf{#1}}\fi
\ifx \doiurl  \undefined \def \doiurl#1{\url{https://doi.org/#1}}\fi
\ifx \betal  \undefined \def \betal{\textit{et al.}}\fi
\ifx \binstitute  \undefined \def \binstitute#1{#1}\fi
\ifx \binstitutionaled  \undefined \def \binstitutionaled#1{#1}\fi
\ifx \bctitle  \undefined \def \bctitle#1{#1}\fi
\ifx \beditor  \undefined \def \beditor#1{#1}\fi
\ifx \bpublisher  \undefined \def \bpublisher#1{#1}\fi
\ifx \bbtitle  \undefined \def \bbtitle#1{#1}\fi
\ifx \bedition  \undefined \def \bedition#1{#1}\fi
\ifx \bseriesno  \undefined \def \bseriesno#1{#1}\fi
\ifx \blocation  \undefined \def \blocation#1{#1}\fi
\ifx \bsertitle  \undefined \def \bsertitle#1{#1}\fi
\ifx \bsnm \undefined \def \bsnm#1{#1}\fi
\ifx \bsuffix \undefined \def \bsuffix#1{#1}\fi
\ifx \bparticle \undefined \def \bparticle#1{#1}\fi
\ifx \barticle \undefined \def \barticle#1{#1}\fi
\bibcommenthead
\ifx \bconfdate \undefined \def \bconfdate #1{#1}\fi
\ifx \botherref \undefined \def \botherref #1{#1}\fi
\ifx \url \undefined \def \url#1{\textsf{#1}}\fi
\ifx \bchapter \undefined \def \bchapter#1{#1}\fi
\ifx \bbook \undefined \def \bbook#1{#1}\fi
\ifx \bcomment \undefined \def \bcomment#1{#1}\fi
\ifx \oauthor \undefined \def \oauthor#1{#1}\fi
\ifx \citeauthoryear \undefined \def \citeauthoryear#1{#1}\fi
\ifx \endbibitem  \undefined \def \endbibitem {}\fi
\ifx \bconflocation  \undefined \def \bconflocation#1{#1}\fi
\ifx \arxivurl  \undefined \def \arxivurl#1{\textsf{#1}}\fi
\csname PreBibitemsHook\endcsname

%%% 1
\bibitem[\protect\citeauthoryear{Gordon et~al.}{2012}]{gordon2012fragmentation}
\begin{barticle}
\bauthor{\bsnm{Gordon}, \binits{M.S.}},
\bauthor{\bsnm{Fedorov}, \binits{D.G.}},
\bauthor{\bsnm{Pruitt}, \binits{S.R.}},
\bauthor{\bsnm{Slipchenko}, \binits{L.V.}}:
\batitle{Fragmentation methods: A route to accurate calculations on large
  systems}.
\bjtitle{Chem.\ Rev.}
\bvolume{112}(\bissue{1}),
\bfpage{632}--\blpage{672}
(\byear{2012})
\doiurl{10.1021/cr200093j}
\end{barticle}
\endbibitem

%%% 2
\bibitem[\protect\citeauthoryear{Fedorov and Kitaura}{2007}]{fedorov2007fmo}
\begin{barticle}
\bauthor{\bsnm{Fedorov}, \binits{D.G.}},
\bauthor{\bsnm{Kitaura}, \binits{K.}}:
\batitle{Extending the power of quantum chemistry to large systems with the
  fragment molecular orbital method}.
\bjtitle{J.\ Phys.\ Chem.~A}
\bvolume{111},
\bfpage{6904}--\blpage{6914}
(\byear{2007})
\doiurl{10.1021/jp0716740}
\end{barticle}
\endbibitem

%%% 3
\bibitem[\protect\citeauthoryear{Hertz and {de Werra}}{1987}]{hertz1987tabu}
\begin{barticle}
\bauthor{\bsnm{Hertz}, \binits{A.}},
\bauthor{\bsnm{{de Werra}}, \binits{D.}}:
\batitle{Using tabu search techniques for graph coloring}.
\bjtitle{Computing}
\bvolume{39}(\bissue{4}),
\bfpage{345}--\blpage{351}
(\byear{1987})
\doiurl{10.1007/BF02239976}
\end{barticle}
\endbibitem

%%% 4
\bibitem[\protect\citeauthoryear{Sherrington and
  Kirkpatrick}{1975}]{sherrington1975solvable}
\begin{barticle}
\bauthor{\bsnm{Sherrington}, \binits{D.}},
\bauthor{\bsnm{Kirkpatrick}, \binits{S.}}:
\batitle{Solvable model of a spin-glass}.
\bjtitle{Physical Review Letters}
\bvolume{35}(\bissue{26}),
\bfpage{1792}--\blpage{1796}
(\byear{1975})
\doiurl{10.1103/PhysRevLett.35.1792}
\end{barticle}
\endbibitem

%%% 5
\bibitem[\protect\citeauthoryear{Kitaura et~al.}{1999}]{kitaura1999fmo}
\begin{barticle}
\bauthor{\bsnm{Kitaura}, \binits{K.}},
\bauthor{\bsnm{Ikeo}, \binits{E.}},
\bauthor{\bsnm{Asada}, \binits{T.}},
\bauthor{\bsnm{Nakano}, \binits{T.}},
\bauthor{\bsnm{Uebayasi}, \binits{M.}}:
\batitle{Fragment molecular orbital method: An approximate computational method
  for large molecules}.
\bjtitle{Chem.\ Phys.\ Lett.}
\bvolume{313},
\bfpage{701}--\blpage{706}
(\byear{1999})
\doiurl{10.1016/S0009-2614(99)00874-X}
\end{barticle}
\endbibitem

%%% 6
\bibitem[\protect\citeauthoryear{Rota}{1964}]{rota1964mobius}
\begin{barticle}
\bauthor{\bsnm{Rota}, \binits{G.-C.}}:
\batitle{On the foundations of combinatorial theory.~{I}. theory of {M\"obius}
  functions}.
\bjtitle{Z.\ Wahrscheinlichkeitstheorie verw.\ Gebiete}
\bvolume{2}(\bissue{4}),
\bfpage{340}--\blpage{368}
(\byear{1964})
\doiurl{10.1007/BF00531932}
\end{barticle}
\endbibitem

%%% 7
\bibitem[\protect\citeauthoryear{Boyd et~al.}{2011}]{boyd2011admm}
\begin{barticle}
\bauthor{\bsnm{Boyd}, \binits{S.}},
\bauthor{\bsnm{Parikh}, \binits{N.}},
\bauthor{\bsnm{Chu}, \binits{E.}},
\bauthor{\bsnm{Peleato}, \binits{B.}},
\bauthor{\bsnm{Eckstein}, \binits{J.}}:
\batitle{Distributed optimization and statistical learning via the alternating
  direction method of multipliers}.
\bjtitle{Found.\ Trends Mach.\ Learn.}
\bvolume{3}(\bissue{1}),
\bfpage{1}--\blpage{122}
(\byear{2011})
\doiurl{10.1561/2200000016}
\end{barticle}
\endbibitem

%%% 8
\bibitem[\protect\citeauthoryear{Dantzig and
  Wolfe}{1960}]{dantzig1960decomposition}
\begin{barticle}
\bauthor{\bsnm{Dantzig}, \binits{G.B.}},
\bauthor{\bsnm{Wolfe}, \binits{P.}}:
\batitle{Decomposition principle for linear programs}.
\bjtitle{Oper.\ Res.}
\bvolume{8}(\bissue{1}),
\bfpage{101}--\blpage{111}
(\byear{1960})
\doiurl{10.1287/opre.8.1.101}
\end{barticle}
\endbibitem

%%% 9
\bibitem[\protect\citeauthoryear{Tseng}{2001}]{tseng2001convergence}
\begin{barticle}
\bauthor{\bsnm{Tseng}, \binits{P.}}:
\batitle{Convergence of a block coordinate descent method for nondifferentiable
  minimization}.
\bjtitle{J.\ Optim.\ Theory Appl.}
\bvolume{109}(\bissue{3}),
\bfpage{475}--\blpage{494}
(\byear{2001})
\end{barticle}
\endbibitem

%%% 10
\bibitem[\protect\citeauthoryear{Shaw}{1998}]{shaw1998lns}
\begin{bchapter}
\bauthor{\bsnm{Shaw}, \binits{P.}}:
\bctitle{Using constraint programming and local search methods to solve vehicle
  routing problems}.
In: \bbtitle{Principles and Practice of Constraint Programming ({CP}'98)}.
\bsertitle{LNCS},
vol. \bseriesno{1520},
pp. \bfpage{417}--\blpage{431}.
\bpublisher{Springer},
\blocation{Berlin, Heidelberg}
(\byear{1998})
\end{bchapter}
\endbibitem

%%% 11
\bibitem[\protect\citeauthoryear{Pisinger and Ropke}{2010}]{pisinger2010lns}
\begin{bchapter}
\bauthor{\bsnm{Pisinger}, \binits{D.}},
\bauthor{\bsnm{Ropke}, \binits{S.}}:
\bctitle{Large neighborhood search}.
In: \beditor{\bsnm{Gendreau}, \binits{M.}},
\beditor{\bsnm{Potvin}, \binits{J.-Y.}} (eds.)
\bbtitle{Handbook of Metaheuristics},
pp. \bfpage{399}--\blpage{419}.
\bpublisher{Springer},
\blocation{Boston, MA}
(\byear{2010})
\end{bchapter}
\endbibitem

%%% 12
\bibitem[\protect\citeauthoryear{Schrimpf et~al.}{2000}]{schrimpf2000ruin}
\begin{barticle}
\bauthor{\bsnm{Schrimpf}, \binits{G.}},
\bauthor{\bsnm{Schneider}, \binits{J.}},
\bauthor{\bsnm{Stamm-Wilbrandt}, \binits{H.}},
\bauthor{\bsnm{Dueck}, \binits{G.}}:
\batitle{Record breaking optimization results using the ruin and recreate
  principle}.
\bjtitle{J.\ Comput.\ Phys.}
\bvolume{159}(\bissue{2}),
\bfpage{139}--\blpage{171}
(\byear{2000})
\end{barticle}
\endbibitem

%%% 13
\bibitem[\protect\citeauthoryear{Mladenovi{\'c} and
  Hansen}{1997}]{mladenovic1997vns}
\begin{barticle}
\bauthor{\bsnm{Mladenovi{\'c}}, \binits{N.}},
\bauthor{\bsnm{Hansen}, \binits{P.}}:
\batitle{Variable neighborhood search}.
\bjtitle{Comput.\ Oper.\ Res.}
\bvolume{24}(\bissue{11}),
\bfpage{1097}--\blpage{1100}
(\byear{1997})
\end{barticle}
\endbibitem

%%% 14
\bibitem[\protect\citeauthoryear{Lauritzen and
  Spiegelhalter}{1988}]{lauritzen1988local}
\begin{barticle}
\bauthor{\bsnm{Lauritzen}, \binits{S.L.}},
\bauthor{\bsnm{Spiegelhalter}, \binits{D.J.}}:
\batitle{Local computations with probabilities on graphical structures and
  their application to expert systems}.
\bjtitle{J.\ R.\ Stat.\ Soc.\ Ser.~B}
\bvolume{50}(\bissue{2}),
\bfpage{157}--\blpage{224}
(\byear{1988})
\doiurl{10.1111/j.2517-6161.1988.tb01721.x}
\end{barticle}
\endbibitem

%%% 15
\bibitem[\protect\citeauthoryear{Bl{\"o}chliger and
  Zufferey}{2008}]{blochliger2008reactive}
\begin{barticle}
\bauthor{\bsnm{Bl{\"o}chliger}, \binits{I.}},
\bauthor{\bsnm{Zufferey}, \binits{N.}}:
\batitle{A graph coloring heuristic using partial solutions and a reactive tabu
  scheme}.
\bjtitle{Comput.\ Oper.\ Res.}
\bvolume{35}(\bissue{3}),
\bfpage{960}--\blpage{975}
(\byear{2008})
\doiurl{10.1016/j.cor.2006.05.014}
\end{barticle}
\endbibitem

%%% 16
\bibitem[\protect\citeauthoryear{Galinier and Hao}{1999}]{galinier1999hybrid}
\begin{barticle}
\bauthor{\bsnm{Galinier}, \binits{P.}},
\bauthor{\bsnm{Hao}, \binits{J.-K.}}:
\batitle{Hybrid evolutionary algorithms for graph coloring}.
\bjtitle{J.\ Combin.\ Optim.}
\bvolume{3}(\bissue{4}),
\bfpage{379}--\blpage{397}
(\byear{1999})
\doiurl{10.1023/A:1009823419804}
\end{barticle}
\endbibitem

%%% 17
\bibitem[\protect\citeauthoryear{Angara et~al.}{2025}]{angara2025penalty}
\begin{botherref}
\oauthor{\bsnm{Angara}, \binits{P.P.}},
\oauthor{\bsnm{Lykov}, \binits{D.}},
\oauthor{\bsnm{Stege}, \binits{U.}},
\oauthor{\bsnm{Alexeev}, \binits{Y.}},
\oauthor{\bsnm{M{\"u}ller}, \binits{H.}}:
The Art of Avoiding Constraints: A Penalty-Free Approach to Constrained
  Combinatorial Optimization with {QAOA}
(2025)
\end{botherref}
\endbibitem

%%% 18
\bibitem[\protect\citeauthoryear{Farhi et~al.}{2014}]{farhi2014qaoa}
\begin{botherref}
\oauthor{\bsnm{Farhi}, \binits{E.}},
\oauthor{\bsnm{Goldstone}, \binits{J.}},
\oauthor{\bsnm{Gutmann}, \binits{S.}}:
A Quantum Approximate Optimization Algorithm
(2014)
\end{botherref}
\endbibitem

%%% 19
\bibitem[\protect\citeauthoryear{Lykov et~al.}{2021}]{lykov2021qtensor}
\begin{bchapter}
\bauthor{\bsnm{Lykov}, \binits{D.}},
\bauthor{\bsnm{Chen}, \binits{A.}},
\bauthor{\bsnm{Chen}, \binits{H.}},
\bauthor{\bsnm{Keipert}, \binits{K.}},
\bauthor{\bsnm{Zhang}, \binits{Z.}},
\bauthor{\bsnm{Gibbs}, \binits{T.}},
\bauthor{\bsnm{Alexeev}, \binits{Y.}}:
\bctitle{Performance evaluation and acceleration of the {QTensor} quantum
  circuit simulator on {GPU}s}.
In: \bbtitle{Proc.\ IEEE/ACM 2nd Int.\ Workshop on Quantum Computing Software
  ({QCS}'21)},
pp. \bfpage{27}--\blpage{34}.
\bpublisher{IEEE},
\blocation{St.~Louis, MO, USA}
(\byear{2021}).
\bcomment{arXiv:2204.06045}
\end{bchapter}
\endbibitem

\end{thebibliography}
\end{document}